\documentclass[12pt,twoside,reqno]{amsproc}

\usepackage[OT1]{fontenc}    %
\usepackage{type1cm}         
\usepackage[english]{babel}  %
\usepackage{amsthm}

\numberwithin{equation}{section}
\numberwithin{figure}{section}

\theoremstyle{plain}
\newtheorem{theorem}{Theorem}[section]

\newtheorem{lemma}[theorem]{Lemma}

\theoremstyle{remark}

\theoremstyle{definition}

\newcommand{\R}{\mathbb{R}}

\newcommand{\N}{\mathbb{N}}

\newcommand{\EE}{\mathcal{E}}

\newcommand{\GG}{\mathcal{G}}
\newcommand{\HH}{\mathcal{H}}

\newcommand{\MM}{\mathcal{M}}

\newcommand{\PP}{\mathcal{P}}
\newcommand{\QQ}{\mathcal{Q}}
\newcommand{\RR}{\mathcal{R}}

\newcommand{\hhh}{\mathtt{h}}
\newcommand{\iii}{\mathtt{i}}
\newcommand{\jjj}{\mathtt{j}}

\newcommand{\dime}{\textrm{dim}_\psi}

\newcommand{\dimh}{\textrm{dim}_H}

\newcommand{\as}{\underline{s}}
\newcommand{\ys}{\overline{s}}

\newcommand{\pallo}{\circ}

\newcommand{\eps}{\varepsilon}

\newcommand{\khii}{\text{\lower -.4ex\hbox{$\chi$}}}

\newcommand{\fii}{\varphi}

\begin{document}

\title{Measures of full dimension on self-affine sets}

\author{Antti K\"aenm\"aki}
\address{Department of Mathematics and Statistics \\
         P.O. Box 35 (MaD) \\
         FIN-40014 University of Jyv\"askyl\"a \\
         Finland}
\email{antakae@maths.jyu.fi}

\date{\today}

\begin{abstract}

Under the assumption of a natural subadditive potential, the so called
cylinder function, working on the symbol space we prove the existence of the
ergodic invariant probability measure satisfying the equilibrium state.
As an application we show that for typical self-affine sets there exists an
ergodic invariant measure having the same Hausdorff dimension as the set itself.
\end{abstract}

\maketitle

\section{Introduction}

In 1981, Hutchinson \cite{hu} introduced the formal definition of
iterated function systems (IFS). Some ideas in this direction have
been presented also earlier, especially in early works of Cantor and
also by Moran \cite{mo} and 
Mandelbrot \cite{man}. Since then self-similar sets, the limit sets
of the so called similitude IFS's, have aroused great interest. By
self-similarity we mean that the set contains copies of itself on
many different scales. Since self-similar sets are widely studied and
with suitable extensions to this setting we can achieve, for example,
so called Julia 
sets, it is interesting to study limit sets of more general systems.

In this note we use one commonly used approach to generalize the
self-similarity. In particular, we consider the sets obtained as geometric
projections of the symbol space.
With the geometric projection here we mean a mapping constructed
as follows. Take a finite set $I$ with at least two elements and
define $I^* = \bigcup_{n=1}^\infty I^n$ and $I^\infty = I^\N$.
Let now $X \subset \R^d$ be a compact set and choose a
collection $\{ X_\iii : \iii \in I^* \}$ of nonempty closed subsets
of $X$ satisfying
\begin{itemize}
  \item[(1)] $X_{\iii,i} \subset X_\iii$ for every $\iii \in I^*$ and
  $i \in I$,
  \item[(2)] $\text{diam}(X_\iii) \to 0$, as $|\iii| \to \infty$.
\end{itemize}
Now the \emph{projection mapping} is the function
$\pi : I^\infty \to X$, for which
\begin{equation*}
  \{ \pi(\iii) \} = \bigcap_{n=1}^\infty X_{\iii|_n}
\end{equation*}
as $\iii \in I^\infty$. The compact set $E = \pi(I^\infty)$ is called a
\emph{limit set}.

In \cite{k} it was introduced the definition of the so called cylinder function
on the symbol space. The main idea of the cylinder function
is to generalize the mass distribution, which is well explained in
Falconer \cite{fa3}. The cylinder function 
gives us a natural subadditive potential and following \cite{k} we are able to 
prove the existence of the ergodic invariant probability measure satisfying
the equilibrium state. Hence working first on the symbol space and then
projecting the situation into $\R^d$ gives us a great amount of flexibility
in studying different
kinds of limit sets. As an application we show that for typical self-affine sets
making a good choice for the cylinder function the geometric projection of the
previously mentioned measure has full Hausdorff dimension.

\section{Measures with full dimension on symbol space}

Let $I$ be a finite set with at least two elements.
Put $I^* = \bigcup_{n=1}^\infty I^n$ and $I^\infty = I^\N = \{
(i_1,i_2,\ldots) : i_j \in I \text{ for } j \in \N \}$. Thus, if
$\iii \in I^*$, there is $k \in \N$ such that $\iii =
(i_1,\ldots,i_k)$, where $i_j \in I$ for all $j=1,\ldots,k$. We
call this $k$ the \emph{length} of $\iii$ and we denote
$|\iii|=k$. If $\jjj \in I^* \cup I^\infty$, then with the
notation $\iii,\jjj$ we mean the element obtained by juxtaposing
the terms of $\iii$ and $\jjj$. If $\iii \in
I^\infty$, we denote $|\iii|=\infty$, and for
$\iii \in I^* \cup I^\infty$ we put $\iii|_k = (i_1,\ldots,i_k)$
whenever $1 \le k < |\iii|$. We define $[\iii;A] = \{ \iii,\jjj : \jjj
\in A \}$ as $\iii \in I^*$ and $A \subset I^\infty$ and we call the set
$[\iii] = [\iii,I^\infty]$ the \emph{cylinder set of level $|\iii|$}.

Define
\begin{equation*}
  |\iii - \jjj| =
  \begin{cases}
    2^{-\min\{ k-1\; :\; \iii|_k \ne \jjj|_k \}}, \quad &\iii \ne \jjj \\
    0, \quad &\iii = \jjj
  \end{cases}
\end{equation*}
whenever $\iii,\jjj \in I^\infty$. Then the couple $(I^\infty,| \cdot |)$ is
a compact metric space. Let us call $(I^\infty,| \cdot |)$ a
\emph{symbol space} and an element $\iii \in I^\infty$ a
\emph{symbol}. If there is no danger of misunderstanding, let us also
call an element $\iii \in I^*$ a symbol. Define the \emph{left shift}
$\sigma : I^\infty \to I^\infty$ by setting
\begin{equation*}
 \sigma(i_1,i_2,\ldots) = (i_2,i_3,\ldots).
\end{equation*}
Clearly, $\sigma$ is continuous and surjective. If $\iii \in I^n$ for
some $n \in \N$, then with the notation $\sigma(\iii)$ we mean the symbol
$(i_2,\ldots,i_n) \in I^{n-1}$.

If for each $t \ge 0$ and $\iii \in I^*$ there is a function $\psi_\iii^t :
I^\infty \to (0,\infty)$, we call it a \emph{cylinder function} if it satisfies 
the following three conditions:
\begin{itemize}
  \item[(1)] There exists $K_t \ge 1$ not depending on $\iii$ such that
  \begin{equation*}
    \psi_\iii^t(\hhh) \le K_t \psi_\iii^t(\jjj)
  \end{equation*}
  for any $\hhh,\jjj \in I^\infty$.
  \item[(2)] For every $\hhh \in I^\infty$ and integer $1 \le j < |\iii|$
  we have
  \begin{equation*}
    \psi_\iii^t(\hhh) \le \psi_{\iii|_j}^t\bigl( \sigma^j(\iii),\hhh \bigr)
	\psi_{\sigma^j(\iii)}^t(\hhh).
  \end{equation*}
  \item[(3)] There exist constants $0<\as,\ys<1$ such that
  \begin{equation*}
    \psi_\iii^t(\hhh)\as^{\delta|\iii|} \le \psi_\iii^{t+\delta}(\hhh)
    \le \psi_\iii^t(\hhh)\ys^{\delta|\iii|}
  \end{equation*}
  for every $\hhh \in I^\infty$.
\end{itemize}
When we speak about one cylinder function, we always assume
there is a collection of them defined as $\iii \in I^*$ and $t \ge 0$.
Let us comment on these conditions. The first one is called the
\emph{bounded variation principle (BVP)}, and it says that the value of
$\psi_\iii^t(\hhh)$ cannot vary too much; roughly speaking,
$\psi_\iii^t$ is essentially constant. The second condition is called
the \emph{subchain rule}. If the subchain rule is satisfied
with equality, we call it a \emph{chain rule}. The third condition is
there just to guarantee the nice behaviour of the cylinder function with
respect to the parameter $t$.

For fixed $\hhh \in I^\infty$, we call the following limit
\begin{equation*}
  P(t) = \lim_{n \to \infty} \tfrac{1}{n} \log \sum_{\iii \in I^n}
  \psi_\iii^t(\hhh)
\end{equation*}
a \emph{topological pressure}. From the definition of the cylinder
function it follows that the topological pressure is continuous,
strictly decreasing and independent of $\hhh$.
We denote the collection of all Borel regular probability measures on
$I^\infty$ with $\MM(I^\infty)$. Define
\begin{equation*}
  \MM_\sigma(I^\infty) = \{ \mu \in \MM(I^\infty) :
                            \mu \text{ is invariant} \},
\end{equation*}
where the invariance of $\mu$ means that $\mu([\iii]) = \mu\bigl(
\sigma^{-1}([\iii]) \bigr)$ for every $\iii \in I^*$.
Now $\MM_\sigma(I^\infty)$ is a 
nonempty closed subset of the compact set $\MM(I^\infty)$ in the weak
topology. For a given $\mu \in \MM_\sigma(I^\infty)$ we define an
\emph{energy} $E_\mu(t)$, by setting
\begin{equation*}
  E_\mu(t) = \lim_{n \to \infty} \tfrac{1}{n} \sum_{\iii \in I^n}
             \mu([\iii]) \log\psi_\iii^t(\hhh)
\end{equation*}
and an \emph{entropy} $h_\mu$ by setting
\begin{equation*}
  h_\mu = -\lim_{n \to \infty} \tfrac{1}{n} \sum_{\iii \in I^n}
          \mu([\iii]) \log\mu([\iii]).
\end{equation*}
It follows that the definition of energy is also independent of $\hhh$.
If we denote
$\alpha(\iii) = \psi_\iii^t(\hhh)/\sum_{\jjj \in I^{|\iii|}}
\psi_\jjj^t(\hhh)$, as $\iii \in I^*$, we get, using Jensen's
inequality for any $n \in \N$ and $\mu \in \MM(I^\infty)$,
\begin{equation} \label{eq:jensen}
\begin{split}
  0 &= 1\log 1 = \tfrac{1}{n} H\Biggl( \sum_{\iii \in I^n} \alpha(\iii)
  \frac{\mu([\iii])}{\alpha(\iii)} \Biggr) \ge \tfrac{1}{n} \sum_{\iii \in I^n}
  \alpha(\iii) H\biggl( \frac{\mu([\iii])}{\alpha(\iii)} \biggr) \\
  &= \tfrac{1}{n} \sum_{\iii \in I^n} \mu([\iii]) \biggl( -\log\mu([\iii]) +
  \log\psi_\iii^t(\hhh) - \log\sum_{\jjj \in I^n} \psi_\jjj^t(\hhh) \biggr)
\end{split}
\end{equation}
with equality if $\mu([\iii]) = \alpha(\iii)$. Here $H(x) = -x\log x$,
as $x>0$, and $H(0)=0$. Thus by letting $n \to \infty$ we conclude
\begin{equation} \label{eq:kaavake}
  P(t) \ge h_\mu + E_\mu(t)
\end{equation}
whenever $\mu \in \MM_\sigma(I^\infty)$. An invariant measure
which satisfies \eqref{eq:kaavake} with equality is called a
\emph{$t$-equilibrium measure}. Furthermore, a measure $\mu \in
\MM(I^\infty)$ is \emph{ergodic} if $\mu(A)=0$ or $\mu(A)=1$ for every
Borel set $A \subset I^\infty$ for which $A = \sigma^{-1}(A)$.

Next, we introduce an important property of functions of the following type.
We say that a function $a : \N \times \N \cup \{ 0 \} \to \R$ satisfies
the \emph{generalized subadditive condition} if
\begin{equation*}
  a(n_1+n_2,0) \le a(n_1,n_2) + a(n_2,0)
\end{equation*}
and $|a(n_1,n_2)| \le n_1C$ for some constant $C$. Furthermore, we say
that this function is \emph{subadditive} if in addition $a(n_1,n_2) =
a(n_1,0)$ for all $n_1 \in \N$ and $n_2 \in \N \cup \{ 0 \}$.

\begin{lemma}[\mbox{\cite[Lemma 2.2]{k}}] \label{thm:gensubadd}
  Suppose that a function $a : \N \times \N \cup \{ 0 \} \to \R$
  satisfies the generalized subadditive condition. Then
  \begin{equation*}
    \tfrac{1}{n} a(n,0) \le \tfrac{1}{kn} \sum_{j=0}^{n-1} a(k,j) +
    \eps(n)
  \end{equation*}
  whenever $0<k<n$. Moreover, if this function is
  subadditive, then the limit $\lim_{n \to \infty} \tfrac{1}{n}
  a(n,0)$ exists and equals to $\inf_n \tfrac{1}{n} a(n,0)$.
\end{lemma}

Here $\eps(n) \downarrow 0$ as $n \to \infty$.
The importance of this lemma in our case is the fact that for any
given $\mu \in \MM(I^\infty)$ the following functions
\begin{align*}
  (n_1,n_2) \mapsto -&\sum_{\iii \in I^{n_1}} \mu \pallo
            \sigma^{-n_2}([\iii]) \log \bigl( \mu \pallo
            \sigma^{-n_2}([\iii]) \bigr) \qquad\text{and} \\
  (n_1,n_2) \mapsto &\sum_{\iii \in I^{n_1}} \mu \pallo
            \sigma^{-n_2}([\iii]) \log\psi_\iii^t(\hhh) + \log K_t
\end{align*}
defined on $\N \times \N \cup \{ 0 \}$ satisfy the generalized
subadditive condition. For the proof, see \cite[Lemma 2.3]{k}.
Notice that if $\mu \in \MM_\sigma(I^\infty)$,
the functions are subadditive. Now, for example, the existence of
the limits in the definition of energy and entropy can be verified easily using
Lemma \ref{thm:gensubadd}. But the real power of this lemma can be seen in the
following theorem.

\begin{theorem}[\mbox{\cite[Theorem 2.6]{k}}] \label{thm:equilibrium}
  There exists an equilibrium measure.
\end{theorem}

\begin{proof}
  For fixed $\hhh \in I^\infty$ we define for each $n \in \N$ a
  probability measure
  \begin{equation*}
    \nu_n = \frac{\sum_{\iii \in I^n} \psi_\iii^t(\hhh)
    \delta_{\iii,\hhh}}{\sum_{\iii \in I^n} \psi_\iii^t(\hhh)},
  \end{equation*}
  where $\delta_\hhh$ is a probability measure with support $\{ \hhh
  \}$. Now with this measure we have equality in \eqref{eq:jensen}.
  Define then for each $n \in \N$ a probability measure
  \begin{equation*}
    \mu_n = \tfrac{1}{n} \sum_{j=0}^{n-1} \nu_n \pallo \sigma^{-j}
  \end{equation*}
  and take $\mu$ to be an accumulation point of the
  set $\{ \mu_n \}_{n \in \N}$ in the weak topology. Now for any $\iii \in
  I^*$ we have
  \begin{align*}
    \bigl|\mu_n([\iii]) - \mu_n\bigl( \sigma^{-1}([\iii]) \bigr)\bigr| &=
    \tfrac{1}{n} \bigl|\nu_n([\iii]) -
    \nu_n\pallo\sigma^{-n}([\iii])\bigr| \\
    &\le \tfrac{1}{n} \to 0,
  \end{align*}
  as $n \to \infty$. Thus $\mu \in \MM_\sigma(I^\infty)$. According to
  Lemma \ref{thm:gensubadd}, we have, using concavity of
  $x \mapsto -x\log x$ as $x>0$,
  \begin{align} \label{eq:seiskatahti}
    -\tfrac{1}{n} \sum_{\iii \in I^n} \nu_n([\iii]) \log \nu_n([\iii]) &\le
    -\tfrac{1}{kn} \sum_{j=0}^{n-1} \sum_{\iii \in I^k} \nu_n
    \pallo \sigma^{-j}([\iii]) \log\bigl( \nu_n
    \pallo \sigma^{-j}([\iii]) \bigr) + \eps(n) \notag \\
    &\le -\tfrac{1}{k} \sum_{\iii \in I^k} \mu_n([\iii]) \log \mu_n([\iii])
    + \eps(n)
  \end{align}
  whenever $0<k<n$. Using Lemma \ref{thm:gensubadd} again, we get also
  \begin{equation} \label{eq:kasitahti}
  \begin{split}
    \tfrac{1}{n} \sum_{\iii \in I^n} \nu_n([\iii]) &\log
    \psi_\iii^t(\hhh) + \tfrac{1}{n} \log K_t \\
    &\le \tfrac{1}{kn} \sum_{j=0}^{n-1} \Biggl( \sum_{\iii \in I^k}
    \nu_n \pallo \sigma^{-j}([\iii]) \log \psi_\iii^t(\hhh) + \log K_t
    \Biggr) + \eps(n) \\
    &= \tfrac{1}{k} \sum_{\iii \in I^k} \mu_n([\iii]) \log
    \psi_\iii^t(\hhh) + \tfrac{1}{k}\log K_t + \eps(n)
  \end{split}
  \end{equation}
  whenever $0<k<n$. Now putting
  \eqref{eq:jensen}, \eqref{eq:seiskatahti} and
  \eqref{eq:kasitahti} together, we have
  \begin{align*}
    \tfrac{1}{n} \log \sum_{\iii \in I^n} \psi_\iii^t(\hhh) &=
    \tfrac{1}{n} \sum_{\iii \in I^n}
    H\bigl( \nu_n([\iii]) \bigr) + \tfrac{1}{n} \sum_{\iii \in I^n}
    \nu_n([\iii]) \log \psi_\iii^t(\hhh) \\
    &\le \tfrac{1}{k} \sum_{\iii \in I^k} H\bigl( \mu_n([\iii]) \bigr)
    + \tfrac{1}{k} \sum_{\iii \in I^k} \mu_n([\iii]) \log
    \psi_\iii^t(\hhh) + \tfrac{1}{k} \log K_t + \eps(n)
  \end{align*}
  whenever $0<k<n$. Letting now $n \to \infty$, we get
  \begin{equation*}
    P(t) \le \tfrac{1}{k} \sum_{\iii \in I^k} H\bigl( \mu([\iii])
    \bigr) + \tfrac{1}{k} \sum_{\iii \in I^k} \mu([\iii]) \log
    \psi_\iii^t(\hhh) + \tfrac{1}{k} \log K_t
  \end{equation*}
  since cylinder sets have empty boundary. The proof is finished by
  letting $k \to \infty$.
\end{proof}

At this point we should emphasize the fact that in general there does not exist a
conformal measure like there does in the case of additive potential, that is,
when the cylinder function satisfies the chain rule.
It is a natural question to ask if the equilibrium measure is ergodic also
in our more general setting, since the lack of conformal measure prevents us to
rely on the proofs used in the case of additive potential. We will answer this
question positively and sketch the proof in the following.

\begin{theorem}[\mbox{\cite[Theorem 4.1]{k}}]
  There exists an ergodic equilibrium measure.
\end{theorem}

\begin{proof}[Sketch of the proof]
  We will consider mappings $\PP, \QQ_n, \QQ : \MM_\sigma(I^\infty) \to \R$,
  for which $\PP(\mu) = h_\mu$, $\QQ_n(\mu) = \tfrac{1}{n} \sum_{\iii \in I^n}
  \mu([\iii]) \log\psi_\iii^t(\hhh)$, and $\QQ(\mu) = E_\mu(t) =
  \lim_{n\to\infty} \QQ_n(\mu)$. It is clear that each $\QQ_n$ is
  affine and continuous 
  (basically because cylinder sets have empty boundary) and $\QQ$ is affine.
  It follows also that $\PP$ is affine and upper semicontinuous.

  Recalling that there exists an equilibrium measure, we assume now contrarily
  that $\PP + \QQ$ cannot attain its supremum with an ergodic measure. Hence
  \begin{equation} \label{eq:antiteesi}
    (\PP + \QQ)(\eta) < (\PP + \QQ)(\mu)
  \end{equation}
  for any ergodic $\eta \in \MM_\sigma(I^\infty)$, where $\mu$ is an equilibrium
  measure. Using Choquet's theorem, we find an ergodic decomposition for any
  invariant measure, that is, for each $\mu \in \MM(I^\infty)$ there exists a
  Borel regular probability measure $\tau_\mu$ on $\EE_\sigma(I^\infty)$,
  the set of all ergodic measures of $\MM_\sigma(I^\infty)$ such that
  \begin{equation} \label{eq:decomposition}
    \RR(\mu) = \int_{\EE_\sigma(I^\infty)} \RR(\eta) d\tau_\mu(\eta)
  \end{equation}
  for every continuous affine $\RR : \MM_\sigma(I^\infty) \to \R$.
  It follows that \eqref{eq:decomposition} remains true also for upper
  semicontinuous affine functions and hence we can write it by using
  $\PP + \QQ_n$. Therefore, applying dominated convergence theorem, we can
  actually write \eqref{eq:decomposition} by using the function $\PP + \QQ$.

  We obtain now the contradiction easily just by integrating
  \eqref{eq:antiteesi} with respect to the measure $\tau_\mu$.
  Denote $A_k = \{ \eta \in \EE_\sigma(I^\infty) : (\PP+\QQ)(\mu)
  - (\PP+\QQ)(\eta) \ge \tfrac{1}{k} \}$, where $\mu$ is an
  equilibrium measure. Now we have $\bigcup_{k=1}^\infty A_k =
  \EE_\sigma(I^\infty)$ and thus $\tau_\mu(A_k) > 0$ for some $k$. Clearly,
  \begin{align*}
    (\PP+\QQ)(\mu) - \int_{\EE_\sigma(I^\infty)} &(\PP+\QQ)(\eta)
    d\tau_\mu(\eta) \\
    &= \int_{\EE_\sigma(I^\infty)} (\PP+\QQ)(\mu) - (\PP+\QQ)(\eta)
    d\tau_\mu(\eta) \\
    &\ge \int_{A_k} \tfrac{1}{k} d\tau_\mu(\eta) = \tfrac{1}{k}
    \tau_\mu(A_k)
  \end{align*}
  for every $k$ and thus
  \begin{equation} \label{eq:erg0}
    (\PP+\QQ)(\mu) > \int_{\EE_\sigma(I^\infty)} (\PP+\QQ)(\eta)
    d\tau_\mu(\eta).
  \end{equation}
\end{proof}

Our aim is to study measures with full dimension. As we work on the symbol
space, we have to first define an appropriate dimension.
For fixed $\hhh \in I^\infty$ we define for each $n \in \N$ 
\begin{equation} \label{eq:falconer_measure1}
  \GG_n^t(A) = \inf\Biggl\{ \sum_{j=1}^\infty
  \psi_{\iii_j}^t(\hhh) : A \subset \bigcup_{j=1}^\infty
  \,[\iii_j],\; |\iii_j| \ge n \Biggr\}
\end{equation}
whenever $A \subset I^\infty$. Assumptions in Carath\'eodory's
construction are now satisfied, and we have a Borel regular
measure $\GG^t$ on $I^\infty$ with
\begin{equation} \label{eq:falconer_measure2}
  \GG^t(A) = \lim_{n \to \infty} \GG_n^t(A).
\end{equation}
Using this measure, we define
\begin{align*}
  \dime(A) &= \inf\{ t \ge 0 : \GG^t(A) = 0 \} \\
           &= \sup\{ t \ge 0 : \GG^t(A) = \infty \},
\end{align*}
and we call this ``critical value'' the \emph{equilibrium dimension}
of the set $A \subset I^\infty$. Observe that due to the BVP the
equilibrium dimension does not depend on the choice of $\hhh$.
We also define the \emph{equilibrium
dimension of a measure $\mu \in \MM(I^\infty)$} by setting $\dime(\mu)
= \inf\{ \dime(A) : A \text{ is a Borel set such that } \mu(A)=1 \}$.

The ergodicity is crucial in studying the equilibrium dimension of the
equilibrium measure $\mu$. According to theorem of Shannon-McMillan and
Kingman's subadditive ergodic theorem we have
\begin{align*}
  h_\mu &= -\lim_{n\to\infty} \tfrac{1}{n} \log([\iii|_n])
    \qquad \text{and} \\
  E_\mu(t) &= \lim_{n\to\infty} \tfrac{1}{n} \log\psi_{\iii|_n}^t(\hhh)
\end{align*}
for $\mu$-almost all $\iii \in I^\infty$. Now from the fact
$P(t)=E_\mu(t)+h_\mu$ we derive
\begin{equation} \label{eq:localdim}
  \lim_{n\to\infty} \frac{\log\mu([\iii|_n])}{\log\psi_{\iii|_n}^t(\hhh)} = 1
\end{equation}
for $\mu$-almost all $\iii \in I^\infty$ provided that $P(t)=0$.
The limit in \eqref{eq:localdim} can be considered as some kind of local
dimension of the equilibrium measure. Now the following theorem is easier
to believe.

\begin{theorem}[\mbox{\cite[Theorem 4.3]{k}}] \label{thm:fulleqdim}
  Suppose $P(t)=0$ and $\mu$ is an ergodic equilibrium
  measure. Then $\dime(\mu) = t$.
\end{theorem}

\section{Application to self-affine sets}

Suppose $X \subset \R^d$ is compact. Let $\{ \fii_\iii : i \in I^* \}$
be a collection of affine mappings of the form
$\fii_i(x) = A_ix+a_i$ for every $i \in I$, where $A_i$ is a contractive
non-singular linear mapping and $a_i \in \R^d$. We also assume that
\begin{itemize}
  \item[(1)] $\fii_{\iii,i}(X) \subset \fii_\iii(X)$ for every $\iii
             \in I^*$ and $i \in I$,
  \item[(2)] $\text{diam}\bigl( \fii_\iii(X) \bigr) \to 0$,
             as $|\iii| \to \infty$.
\end{itemize}
By \emph{contractivity} we mean that for every $i
\in I$ there exists a constant $0< s_i <1$ such that
$|\fii_i(x)-\fii_i(y)| \le
s_i |x-y|$ whenever $x,y \in \R^d$. This collection is called
an \emph{affine iterated function system} and the
corresponding limit set a \emph{self-affine set}.

Clearly, the products $A_\iii =
A_{i_1}\cdots A_{i_{|\iii|}}$ are also contractive and
non-singular. Singular values of a non-singular matrix are the
lengths of the principle semiaxes of the image of the unit ball. On
the other hand, the singular values $1 > \alpha_1 \ge \alpha_2 \ge
\cdots \ge \alpha_d > 0$ of a contractive, non-singular matrix $A$
are the non-negative square roots of the eigenvalues of $A^*A$,
where $A^*$ is the transpose of $A$. Define a singular value function
$\alpha^t$ by setting $\alpha^t(A) =
\alpha_1\alpha_2\cdots\alpha_{l-1}\alpha_l^{t-l+1}$, where $l$ is
the smallest integer greater than $t$ or equal to it. For all $t>d$ we put
$\alpha^t(A) = (\alpha_1\cdots\alpha_d)^{t/d}$. It is clear that
$\alpha^t(A)$ is continuous and strictly decreasing in $t$. If for each
$\iii \in I^*$ we choose $\psi_\iii^t \equiv \alpha^t(A_\iii)$, then
$\psi_\iii^t$ is a constant cylinder function. The subchain rule for
$\psi_\iii^t$ is satisfied by Falconer \cite[Lemma 2.1]{fa1}.
We call this choice of a cylinder function in this setting a
\emph{natural cylinder function}.

Our aim is to study the Hausdorff dimension of
measures on self-affine sets. We
say that the Hausdorff dimension of a given Borel probability measure
$m$ is $\dimh(m) = \inf\{ \dimh(A) : A \text{ is a Borel set such that
} m(A)=1 \}$. Checking whether $\dimh(m) = \dimh(E)$ is one way to examine
how well a given measure $m$ is spread out on a given set $E$.
The desired result follows from Theorem \ref{thm:fulleqdim}
by applying Falconer's result for the Hausdorff dimension of
self-affine sets; see \cite{fa1}.

\begin{theorem}[\mbox{\cite[Theorem 4.5]{k}}]
  Suppose an affine IFS is equipped with the natural cylinder function
  and the mappings are of the form $\fii_i(x) = A_ix + a_i$, where
  $|A_i| < \tfrac{1}{2}$. We also assume that $P(t)=0$ and $m$ is a
  projected $t$-equilibrium measure. Then for $\HH^{d\#I}$-almost
  all $a=(a_1,\ldots,a_{\#I}) \in \R^{d\#I}$ we have
  \begin{equation*}
    \dimh(m) = \dimh(E),
  \end{equation*}
  where $E = E(a)$.
\end{theorem}

\begin{proof}
  According to Theorem \ref{thm:fulleqdim} we have $\dime(A) =
  \dime(I^\infty)$ whenever $A \subset I^\infty$ has full $\mu$-measure.
  Hence for any $A \subset E$ with full $m$-measure we get
  \begin{equation*}
    \dimh(A) = \dime\bigl( \pi^{-1}(A) \bigr) = \dime(I^\infty) = \dimh(E)
  \end{equation*}
  using Falconer's theorem.
\end{proof}

This theorem gives also a
partially positive answer to the open question proposed by Kenyon and
Peres in \cite{kp}. Partially positive in a sense that we are able to extend
the class of those sets for which we can give a positive answer
in their question. They asked whether there exists a $T$-invariant
ergodic probability measure on a given compact set, where the mapping
$T$ is continuous and expanding, such that it has full dimension. In
our case the mapping $T$ on the self-affine set is constructed by using
inverses of the mappings of IFS.


\end{document}